\documentclass[12pt, reqno]{amsart}
\usepackage{amsmath, amsthm, amscd, amsfonts, amssymb, graphicx, color}
\usepackage[bookmarksnumbered, colorlinks, plainpages]{hyperref}

\makeatletter \oddsidemargin.9375in \evensidemargin \oddsidemargin
\def\Speaker{$^{*}$\protect\footnotetext{ \lowercase{}}}
\def\authorsaddresses#1{\dedicatory{#1}}
\newtheorem{theorem}{Theorem}[section]
\newtheorem{lemma}[theorem]{Lemma}

\newtheorem{corollary}[theorem]{Corollary}
\theoremstyle{definition}

\theoremstyle{remark}

\numberwithin{equation}{section}

\begin{document}
\setcounter{page}{1}
%%%%%%%%%%%%%%%%%%%%%%%%%%%%%%%%%%%%%%%%%%%%%

\title[A fibering map approach for a Laplacian system ... ]{}{\textbf{A Fibering Map Approach for a Laplacian System With \\\begin{center}Sign-Changing Weight Function\end{center}}

\author[ S. S. Kazemipoor AND M. Zakeri ] { Seyyed Sadegh Kazemipoor\Speaker And Mahboobeh Zakeri}
\authorsaddresses{Department of Mathematical Analysis, Charles University, Prague, Czech Republic.\\
}
\thanks{\subjclass MMSC[2010] : {Primary 35J20; Secondary 35J50, 35J91.}}
%-------------------------------------------------------------------------------
\begin{abstract}
We prove the existence of at least two positive solutions for the
Laplacian system
$$\left\{\begin{array}{ll}
-\Delta u=\lambda a(x)|u|^{q-2}u+\frac{\alpha}{\alpha+\beta}b(x)|u|^{\alpha-2}u|v|^{\beta}&$for~$x\in\Omega$$,  \\
-\Delta v=\lambda a(x)|v|^{q-2}v+\frac{\beta}{\alpha+\beta}b(x)|u|^{\alpha}|v|^{\beta-2}v&$for~$x\in\Omega$$,\qquad(E_{\lambda}) \\
u=v=0 &$for~$x\in\partial\Omega.$$
\end{array}\right.$$
On a bounded region $\Omega$ by using the Nehari manifold and the
fibering maps associated with the Euler functional for the system.
\end{abstract}
\maketitle
%%%%%%%%%%%%%%%%%%%%%%%
\section{Introduction}We shall discuss the
existence of positive solutions of the Laplacian system
\[\left\{\begin{array}{lll}
-\Delta u=\lambda a(x)|u|^{q-2}u+\frac{\alpha}{\alpha+\beta}b(x)|u|^{\alpha-2}u|v|^{\beta}&$for~$x\in\Omega$$,  \\
-\Delta v=\lambda a(x)|v|^{q-2}v+\frac{\beta}{\alpha+\beta}b(x)|u|^{\alpha}|v|^{\beta-2}v&$for~$x\in\Omega$$,\qquad(E_{\lambda}) \\
u=v=0 &$for~$x\in\partial\Omega$$
\end{array}\right.\]
\\
Where $\Omega$ is a bounded region with smooth boundary in $\Bbb
R^{N}$.\\
$\alpha>1,~~\beta>1$ satisfying $\alpha+\beta<2^{*}$
($2^{*}=\frac{2N}{N-2}\quad if~N>2, \quad 2^{*}=\infty\quad if
~N\leq 2$), $\lambda>0$, $1<q<2,~~ a\in L^{p^{*}}(\Omega)$ where
$p^{*}=\frac{\alpha+\beta}{\alpha+\beta-q}$ and
$a,b:\Omega\rightarrow\Bbb{R}$ are smooth
functions which are somewhere positive but which may change sign on $\Omega$.\\
Set $\alpha=\beta=\frac{p}{2}, u=v$. Then problem
$(E_({\lambda}_))$ reduces to the semilinear scalar elliptic
equations with concave-convex nonlineariities:
\[\left\{\begin{array}{llll}
-\Delta u=\lambda f(x)|u|^{q-2}u+h(x)|u|^{p-2}u&$in~$\Omega$$,&\hspace{2cm}(E_({\lambda}_))& \\
u=0 &$in~$\partial\Omega$$.
\end{array}\right.\]
For $2<p<2^{*}$ and the weight functions $f\equiv h\equiv 1$, the
authors Ambrosetti-Brezis-Cerami \cite{4}have investigated Eq.
$(E_({\lambda}_))$ in \cite{4}. They found that there exists
$\lambda_{0}>0$ such that Eq. $(E_({\lambda}_))$ admits at least two
positive solutions for $\lambda\in(0,\lambda_{0})$, has a positive
solution for $\lambda=\lambda_{0}$ and no positive solution
exists for $\lambda>\lambda_{0}$. For more general cases, we
refer the reader to Afrouzi-Ala-Kazemipoor \cite{1}, Ambrosetti-Azorezo-Peral \cite{3}, de
Figueiredo-Gossez-Ubilla \cite{13}, EL Hamidi \cite{15} and Wu \cite{25}, etc.
Recently, in \cite{25} the author has considered a semilinear scalar
elliptic equation involving concave-convex nonlinearities and
sign-changing weight function, and showed multiplicity results
with respect to the parameter $\lambda$ via the extraction of
Palais-Smale sequences in the Nehari manifold, where for the
definition of Nehari
manifolds we refer the reader to Nehari \cite{18} or Willem \cite{23}.\\
For the semilinear elliptic systems with concave-convex
nonlinearities, the authors Adriouch- EL Hamidi \cite{5} considered the
following problems:
\[\left\{\begin{array}{lll}
-\Delta u=\lambda u+\frac{\alpha}{\alpha+\beta}|u|^{\alpha-2}u|v|^{\beta}&$in~$\Omega$$,  \\
-\Delta v=\mu |v|^{q-2}v+\frac{\beta}{\alpha+\beta}|u|^{\alpha}|v|^{\beta-2}v&$in~$\Omega$$, \\
u=v=0 &$on~$\partial\Omega.$$
\end{array}\right.\]
They proved this system has at least two positive solutions when
the pair of the parameters $(\lambda,\mu)$ belongs to a certain
subset of $\Bbb R^{2}$. For more similar problems, we refer the
reader to Ahammou \cite{2}, Alves-de Morais Filho-Souto \cite{6},
Bozhkov-Mitidieri \cite{8}, Cl$\acute{e}$ment-de Figueiredo-Mitidieri
\cite{11}, de Figueiredo-Felmer \cite{12}, El Hamidi \cite{16}, Squassina \cite{19}
and v$\acute{e}$lin
\cite{22}, etc.\\
By the above results we know that the existence and multiplicity
of positive solutions of semilinear elliptic problems depend on
the nonlinearity term . Let $S$ be the best Sobolev constant for the embedding of
$H_{0}^{1}(\Omega)$ in
$L^{\alpha+\beta}(\Omega)$. Then we have the following result.\\
In this paper, we give a very simple variational proof which is
similar to proof
of Brown-Wu (see \cite{9}) to prove the existence of at least two positive solutions of system $(E_{\lambda})$ which is similar to the Wu system \cite{26}.\\
We shall throughout use the function space
$W=W_{0}^{1,2}(\Omega)\times W_{0}^{1,2}(\Omega)$ with norm
$$||(u,v)||=(\int_{\Omega}|\nabla u|^{2}dx+\int_{\Omega}|\nabla v|^{2}dx)^{\frac{1}{2}}$$
and the standard $L^{p}(\Omega)\times L^{p}(\Omega)$ spaces whose
norms we denote by $||(u,v)||_{p}$.\\
\section{Fibering Maps and The Nehari Manifold}
The Euler functional associated with $(E_{\lambda})$ is
$$\begin{array}{rcl}J_{\lambda}(u,v)&=&\frac{1}{2}(\int_{\Omega}|\nabla u|^{2}dx+\int_{\Omega}|\nabla
v|^{2}dx)\\
&-&\frac{\lambda}{q}(\int_{\Omega}a(x)|u|^{q}dx+\int_{\Omega}a(x)|v|^{q}dx)\\
&-&\frac{1}{\alpha+\beta}(\int_{\Omega}b(x)|u|^{\alpha}|v|^{\beta}dx)\\
\end{array}$$
for all $(u,v)\in W$.\\
As $J_{\lambda}$ is not bounded below on $W$, it is useful to
consider the functional on the Nehari manifold
$$M_{\lambda}(\Omega)=\{(u,v)\in W~:~~\langle J^{'}_{\lambda}(u,v),(u,v)\rangle=0\}$$
where $\langle~,~\rangle$ denotes the usual duality. Thus
$(u,v)\in M_{\lambda}(\Omega)$ if and only if
$$\begin{array}{rcl}&&(\int_{\Omega}|\nabla u|^{2}dx+\int_{\Omega}|\nabla
v|^{2}dx)-\lambda(\int_{\Omega}a(x)|u|^{q}dx+\int_{\Omega}a(x)|v|^{q}dx)\\
&-&(\int_{\Omega}b(x)|u|^{\alpha}|v|^{\beta}dx)=0.\hspace{4cm}(2.1)\\
\end{array}$$
Clearly $M_{\lambda}(\Omega)$ is a much smaller set than $W$ and, as
we shall show, $J_{\lambda}$ is much better behaved on
$M_{\lambda}(\Omega)$. In particular, on $M_{\lambda}(\Omega)$ we
have that
$$\begin{array}{rcl}J_{\lambda}(u,v)&=&(\frac{1}{2}-\frac{1}{q})(\int_{\Omega}|\nabla u|^{2}dx+\int_{\Omega}|\nabla
v|^{2}dx)\\
&+&(\frac{1}{q}-\frac{1}{\alpha+\beta})(\int_{\Omega}b(x)|u|^{\alpha}|v|^{\beta}dx)\\
&=&(\frac{1}{2}-\frac{1}{\alpha+\beta})(\int_{\Omega}|\nabla u|^{2}dx+\int_{\Omega}|\nabla
v|^{2}dx)\\
&-&\lambda(\frac{1}{q}-\frac{1}{\alpha+\beta})(\int_{\Omega}a(x)|u|^{q}dx+\int_{\Omega}a(x)|v|^{q}dx).\\
\end{array}\hspace{1cm}(2.2)$$\
%%%%%%%%%%%%%%%%%%%%%%%%%%%%%%%%%%%%%%%%%%%%%%%%%%%%%%%%%%%%%%%%%%%%%%

\begin{theorem} $J_{\lambda}$ is coercive and bounded below on $M_{\lambda}(\Omega)$.
\end{theorem}
\begin{proof}[\bf Proof]
It follows from (2.2) and by the H$\ddot{o}$lder inequality
$$\begin{array}{rcl}J_{\lambda}(u,v)&=&\frac{\alpha+\beta-2}{2(\alpha+\beta)}||(u,v)||^{2}_{H}-(\frac{\alpha+\beta-q}{q(\alpha+\beta)})(\int_{\Omega}\lambda a(x)|u|^{q}+\int_{\Omega}\lambda|v|^{q}dx)\\
&\geq&\frac{\alpha+\beta-2}{2(\alpha+\beta)}||(u,v)||^{2}_{H}\\
&-&S^{q}(\frac{\alpha+\beta-q}{q(\alpha+\beta)})((|\lambda|||a||_{L^{p^{*}}})^{\frac{2}{2-q}}+(|\lambda|||a||_{L^{p^{*}}})^{\frac{2}{2-q}})^{\frac{2-q}{2}}||(u,v)||^{q}_{H}\\
&\geq&\frac{S^{\frac{2q}{2-q}}(q-2)(\alpha+\beta-q)^{\frac{2}{2-q}}}{2q(\alpha+\beta)(\alpha+\beta-2)^{\frac{q}{2-q}}}((|\lambda|||a||_{L^{p^{*}}})^{\frac{2}{2-q}}+(|\lambda|||a||_{L^{p^{*}}})^{\frac{2}{2-q}}).\\
\end{array}$$
Thus, $J_{\lambda}$ is coercive on $M_{\lambda}(\Omega)$ and
$$J_{\lambda}(u,v)\geq\frac{S^{\frac{2q}{2-q}}(q-2)(\alpha+\beta-q)^{\frac{2}{2-q}}}{2q(\alpha+\beta)(\alpha+\beta-2)^{\frac{q}{2-q}}}((|\lambda|||a||_{L^{p^{*}}})^{\frac{2}{2-q}}+(|\lambda|||a||_{L^{p^{*}}})^{\frac{2}{2-q}}).$$
This completes the proof.
\end{proof}
The Nehari manifold is closely linked to the behaviour of the
functions of the form $\phi_{u,v}:t\rightarrow
J_{\lambda}(tu,tv)\quad(t>0)$. Such maps are known as fibering
maps and were introduced by Drabek and Pohozaev in \cite{14} and are
also discussed in Brown and Zhang \cite{10}. If $(u,v)\in W$, we have
$$\begin{array}{rcl}\phi_{u,v}(t)&=&\frac{1}{2}t^{2}(\int_{\Omega}|\nabla u|^{2}dx+\int_{\Omega}|\nabla
v|^{2}dx)\\
&-&\lambda\frac{t^{q}}{q}(\int_{\Omega}a(x)|u|^{q}dx+\int_{\Omega}a(x)|v|^{q}dx)\\
&-&\frac{t^{\alpha+\beta}}{\alpha+\beta}(\int_{\Omega}b(x)|u|^{\alpha}|v|^{\beta}dx)\\
\end{array}\hspace{2cm}(2.3)$$

$$\begin{array}{rcl}\phi_{u,v}^{'}(t)&=&t (\int_{\Omega}|\nabla u|^{2}dx+\int_{\Omega}|\nabla
v|^{2}dx)\\
&-&\lambda t^{q-1}(\int_{\Omega}a(x)|u|^{q}dx+\int_{\Omega}a(x)|v|^{q}dx)\\
&-&t^{\alpha+\beta-1}(\int_{\Omega}b(x)|u|^{\alpha}|v|^{\beta}dx)\\
\end{array}\hspace{2cm}(2.4)$$

$$\begin{array}{rcl}\phi_{u,v}^{''}(t)&=&(\int_{\Omega}|\nabla u|^{2}dx+\int_{\Omega}|\nabla
v|^{2}dx)\\
&-&(q-1)\lambda t^{q-2}(\int_{\Omega}a(x)|u|^{q}dx+\int_{\Omega}a(x)|v|^{q}dx)\\
&-&(\alpha+\beta-1)t^{(\alpha+\beta-2)}(\int_{\Omega}b(x)|u|^{\alpha}|v|^{\beta}dx).\\
\end{array}\hspace{2cm}(2.5)$$
It is easy to see that $(u,v)\in M_{\lambda}(\Omega)$ if and only
if $\phi_{u,v}^{'}(1)=0$ and, more generally, that
$\phi_{u,v}^{'}(t)=0$ if and only if $(tu,tv)\in
M_{\lambda}(\Omega)$, i.e., elements in $M_{\lambda}(\Omega)$
correspond to stationary points of fibering maps. Thus it is
natural to subdivide $M_{\lambda}(\Omega)$ into sets
corresponding to local minima, local maxima and points of
inflection and so we define
$$\begin{array}{rcl}M_{\lambda}^{+}(\Omega)&=&\{(u,v)\in
M_{\lambda}(\Omega)~:~\phi_{u,v}^{''}(1)>0\},\\
M_{\lambda}^{-}(\Omega)&=&\{(u,v)\in
M_{\lambda}(\Omega)~:~\phi_{u,v}^{''}(1)<0\},\\
M_{\lambda}^{0}(\Omega)&=&\{(u,v)\in
M_{\lambda}(\Omega)~:~\phi_{u,v}^{''}(1)=0\},\\
\end{array}$$
and note that if  $(u,v)\in M_{\lambda}(\Omega)$, i.e.,
$\phi_{u,v}^{'}(1)=0$, then
$$\begin{array}{rcl}\phi_{u,v}^{''}(1)&=&(2-q)(\int_{\Omega}|\nabla u|^{2}dx+\int_{\Omega}|\nabla
v|^{2}dx)\\
&-&(\alpha+\beta-q)(\int_{\Omega}b(x)|u|^{\alpha}|v|^{\beta}dx)\\
&=&(\alpha+\beta+2)(\int_{\Omega}|\nabla u|^{2}dx+\int_{\Omega}|\nabla
v|^{2}dx)\\
&-&\lambda(q-\alpha-\beta)(\int_{\Omega}a(x)|u|^{q}dx+\int_{\Omega}a(x)|v|^{q}dx).\\
\end{array}\hspace{2cm}(2.6)$$
Also, as proved in Binding, Drabek and Huang \cite{7},Wu \cite{24} or in Brown and
Zhang \cite{10}, we have the following lemma. \begin{theorem} Suppose that
$(u_{0},v_{0})$ is a local maximum or minimum for $J_{\lambda}$ on
$M_{\lambda}(\Omega)$. Then, if $(u_{0},v_{0})\notin
M_{\lambda}^{0}(\Omega)$, $(u_{0},v_{0})$ is a critical point of
$J_{\lambda}$.
\end{theorem}
\section{Analysis of The Fibering Maps}
It this section we give a fairly complete description of the
fibering maps associated with the system. As we shall see the
essential nature of the maps is determined by the signs of
$(\int_{\Omega}a(x)|u|^{q}dx+\int_{\Omega}a(x)|v|^{q}dx)$ and
$(\int_{\Omega}b(x)|u|^{\alpha}|v|^{\beta}dx)$ we will find it
useful to consider the function
$$\begin{array}{rcl}m_{u,v}(t)&=&t^{(2-q)}(\int_{\Omega}|\nabla u|^{2}dx+\int_{\Omega}|\nabla
v|^{2}dx)\\
&-&t^{(\alpha+\beta-q+1)}(\int_{\Omega}b(x)|u|^{\alpha}|v|^{\beta}dx).\\
\end{array}$$
Clearly, for $t>0$, $(tu,tv)\in M_{\lambda}(\Omega)$ if and only
if $t$ is a solution of
$$m_{u,v}(t)=\lambda(\int_{\Omega}a(x)|u|^{q}dx+\int_{\Omega}a(x)|v|^{q}dx).\hspace{2cm}(3.1)$$
Moreover,
$$\begin{array}{rcl}m_{u,v}^{'}(t)&=&(2-q)t^{(1-q)}(\int_{\Omega}|\nabla u|^{2}dx+\int_{\Omega}|\nabla
v|^{2}dx)\\
&-&(\alpha+\beta-q+1)t^{(\alpha+\beta-q)}(\int_{\Omega}b(x)|u|^{\alpha}|v|^{\beta}dx).\\
\end{array}\hspace{2cm}(3.2)$$
It is easy to see that $m_{u,v}$ is a strictly  increasing
function for $t\geq0$ whenever
$(\int_{\Omega}b(x)|u|^{\alpha}|v|^{\beta}dx)\leq0$ and $m_{u,v}$
is initially increasing and eventually decreasing with a single
turning point when
$(\int_{\Omega}b(x)|u|^{\alpha}|v|^{\beta}dx)>0$.\\
Suppose $(tu,tv)\in M_{\lambda}(\Omega)$. It follows from (2.6) and
(3.2) that\\ $\phi_{tu,tv}^{''}(1)=t^{q+1}m_{u,v}^{'}(t)$ and so
$(tu,tv)\in M_{\lambda}^{+}(\Omega)~(M_{\lambda}^{-}(\Omega))$
provided $m^{'}_{u,v}(t)>0~(<0)$.~We shall now describe the nature
of the fibering maps for all possible signs of
$(\int_{\Omega}b(x)|u|^{\alpha}|v|^{\beta}dx)$ and
$(\int_{\Omega}a(x)|u|^{q}dx+\int_{\Omega}a(x)|v|^{q}dx)$. If
$(\int_{\Omega}b(x)|u|^{\alpha}|v|^{\beta}dx)\\\leq0$ and
$(\int_{\Omega}a(x)|u|^{q}dx+ \int_{\Omega}a(x)|v|^{q}dx)\leq0$,
clearly $\phi_{u,v}$ is an increasing function of $t$; thus in this
case no multiple of $(u,v)$ lies in $M_{\lambda}(\Omega)$. If
$(\int_{\Omega}b(x)|u|^{\alpha}|v|^{\beta}dx)\leq0$ and
$(\int_{\Omega}a(x)|u|^{q}dx+ \int_{\Omega}a(x)|v|^{q}dx)>0$, then
$m_{u,v}$ is clear that there is exactly one solution of $u=0, v=0$
for $x\in\partial\Omega$. Thus there is a unique value $t(u,v)>0$
such that $(t(u,v)u,t(u,v)v)\in M_{\lambda}(\Omega)$. Clearly
$m_{u,v}^{'}(t(u,v))>0$ and so $(t(u,v)u,\\t(u,v)v)\in
M_{\lambda}^{+}(\Omega)$. Thus the fibering map $\phi_{u,v}$ has
a unique critical point at $t=t(u,v)$ which is a local minimum.\\
Suppose now $(\int_{\Omega}b(x)|u|^{\alpha}|v|^{\beta}dx)>0$ and
$(\int_{\Omega}a(x)|u|^{q}dx\\+\int_{\Omega}a(x)|v|^{q}dx)\leq0$. Then
$m_{u,v}$ is clear that there is exactly one positive solution of
$u=0, v=0$ for $x\in\partial\Omega$. Thus there is again a unique
value $t(u,v)>0$ such that $(t(u,v)u, t(u,v)v)\in
M_{\lambda}(\Omega)$ and since $m^{'}_{u,v}(t(u,v))<0$ in this case
$(t(u,v)u, t(u,v)v)\in M_{\lambda}^{-}(\Omega)$. Hence the fibering
map $\phi_{u,v}$ has a unique  critical point which is a local
maximum.\\Finally we consider the case
$(\int_{\Omega}b(x)|u|^{\alpha}|v|^{\beta}dx)>0$ and
$(\int_{\Omega}a(x)|u|^{q}dx+\int_{\Omega}a(x)|v|^{q}dx)>0$ where
the situation is more complicated. If $\lambda>0$ is sufficiently
large, $u=0, v=0$ for $x\in\partial\Omega$ has no solution and so
$\phi_{u,v}$ has no critical points - in this case $\phi_{u,v}$ is a
decreasing function. Hence no multiple of $(u,v)$ lies in
$M_{\lambda}(\Omega)$. If, on the other hand, $\lambda>0$ is
sufficiently small, there are exactly two solutions
$t_{1}(u,v)<t_{2}(u,v)$ of $u=0, v=0$ for $x\in\partial\Omega$ with
$m^{'}_{u,v}(t_{1}(u,v))>0$
and $m^{'}_{u,v}(t_{2}(u,v))<0$. \\
Thus there are exactly two multiples of $(u,v)\in
M_{\lambda}(\Omega)$, namely\\$(t_{1}(u,v)u,t_{1}(u,v)v)\in
M_{\lambda}^{+}(\Omega)$ and $(t_{2}(u,v)u,t_{2}(u,v)v)\in
M_{\lambda}^{-}(\Omega)$. It follows that $\phi_{u,v}$ has
exactly two critical points- a local minimum at $t=t_{1}(u,v)$
and a local maximum at $t=t_{2}(u,v)$; moreover $\phi_{u,v}$ is
decreasing in $(0,t_{1})$, increasing in $(t_{1},t_{2})$ and
decreasing in $(t_{2},\infty)$.\\
\begin{lemma} There exists
$\lambda_{1}>0$ such that, when $\lambda<\lambda_{1}$,
$\phi_{u,v}$ takes on positive values for all non-zero $(u,v)\in
W$.
\end{lemma}
\begin{proof}[\bf Proof]
If $(\int_{\Omega}b(x)|u|^{\alpha}|v|^{\beta}dx)\leq0$, then
$\phi_{u,v}(t)>0$ for $t$ sufficiently large. Suppose $(u,v)\in
W$ and $(\int_{\Omega}b(x)|u|^{\alpha}|v|^{\beta}dx)>0$. Let
$$\begin{array}{rcl}h_{u,v}(t)&=&\frac{t^{2}}{2}(\int_{\Omega}|\nabla u|^{2}dx+\int_{\Omega}|\nabla
v|^{2}dx)\\\\
&-&\frac{t^{(\alpha+\beta)}}{\alpha+\beta}(\int_{\Omega}b(x)|u|^{\alpha}|v|^{\beta}dx).\\
\end{array}$$
Then elementary calculus shows that $h_{u,v}$ takes on a maximum
value of
$$\frac{\alpha+\beta-2}{2q}\{\frac{(\int_{\Omega}|\nabla u|^{2}dx+\int_{\Omega}|\nabla
v|^{2}dx)^{(\alpha+\beta)}}{(\int_{\Omega}b(x)|u|^{\alpha}|v|^{\beta}dx)^{2}}\}^{\frac{1}{\alpha+\beta-2}}$$
when $$t=t_{\max}=(\frac{\int_{\Omega}|\nabla
u|^{2}dx+\int_{\Omega}|\nabla
v|^{2}dx}{\int_{\Omega}b(x)|u|^{\alpha}dx+\int_{\Omega}b(x)|v|^{\beta}dx})^{\frac{1}{\alpha+\beta-2}}.$$
However
$$\frac{(\int_{\Omega}|\nabla u|^{2}dx+\int_{\Omega}|\nabla
v|^{2}dx)^{(\alpha+\beta)}}{(\int_{\Omega}|u|^{\alpha}|v|^{\beta}dx)^{2}}\geq\frac{1}{S_{\alpha+\beta}^{2(\alpha+\beta)}}$$

where $S_{\alpha+\beta}$ denotes the Sobolev constant of the
embedding of $W$ into $L=L^{\alpha+\beta}(\Omega)\times
L^{\alpha+\beta}(\Omega)$.
$$h_{u,v}(t_{\max})\geq\frac{\alpha+\beta-2}{2(\alpha+\beta)}(\frac{1}{||b^{+}||^{2}_{\infty}S_{\alpha+\beta}^{2(\alpha+\beta)}})^{\frac{1}{\alpha+\beta-2}}=\delta$$
where $\delta$ is independent of $u$ and $v$.\\
We shall now show that there exists $\lambda_{1}>0$ such that
$\phi_{u,v}(t_{\max})>0$, i.e.,
$$h_{u,v}(t_{\max})-\frac{\lambda(t_{\max})^{q}}{q}(\int_{\Omega}a(x)|u|^{q}dx+\int_{\Omega}a(x)|v|^{q}dx)>0$$
for all $(u,v)\in W-\{(0,0)\}$ provided $\lambda<\lambda_{1}$. We
have
$$\begin{array}{rcl}&&\frac{(t_{\max})^{q}}{q}(\int_{\Omega}a(x)|u|^{q}dx+\int_{\Omega}a(x)|v|^{q}dx)\\\\
&\leq&\frac{1}{q}||a||_{\infty}S_{q}^{q}(\frac{\int_{\Omega}|\nabla u|^{2}dx+\int_{\Omega}|\nabla
v|^{2}dx}{\int_{\Omega}b(x)|u|^{\alpha}|v|^{\beta}dx})^{\frac{q}{\alpha+\beta-2}}\\\\
&&(\int_{\Omega}|\nabla u|^{2}dx+\int_{\Omega}|\nabla
v|^{2}dx)^{\frac{q}{2}}\\\\
&=&\frac{1}{q}||a||_{\infty}S_{q}^{q}\{\frac{(\int_{\Omega}|\nabla u|^{2}dx+\int_{\Omega}|\nabla
v|^{2}dx)^{(\alpha+\beta)}}{(\int_{\Omega}b(x)|u|^{\alpha}|v|^{\beta}dx)^{2}}\}^{\frac{q-2}{2(\alpha+\beta-2)}}\\\\
&=&\frac{1}{q}||a||_{\infty}S_{q}^{q}[\frac{2(\alpha+\beta)}{\alpha+\beta-2}]^{\frac{q}{2}}h_{u,v}(t_{\max})^{\frac{q}{2}}\\\\
&=&ch_{u,v}(t_{\max})^{\frac{q}{2}}\\\\
\end{array}$$
where $c$ is independent of $u$ and $v$. Hence
$$\begin{array}{rcl}\phi_{u,v}(t_{\max})&\geq&h_{u,v}(t_{\max})-\lambda
ch_{u,v}(t_{\max})^{\frac{q}{2}}\\\\
&=&h_{u,v}(t_{\max})^{\frac{q}{2}}(h_{u,v}(t_{\max})^{\frac{2-q}{2}}-\lambda
c)\\\\
\end{array}$$
and so, since $h_{u,v}(t_{\max})\geq\delta$ for all $(u,v)\in
W-\{(0,0)\}$, it follows that $\phi_{u,v}(t_{\max})\geq0$ for all
non-zero $(u,v)$ provided
$\lambda<\delta^{\frac{2-q}{2}}/2c=\lambda_{1}$. This completes
the proof.
\end{proof}
It follows from the above lemma that when $\lambda<\lambda_{1}$,
$(\int_{\Omega}a(x)|u|^{q}dx+\int_{\Omega}a(x)|v|^{q}dx)\\>0$ and
$(\int_{\Omega}b(x)|u|^{\alpha}|v|^{\beta}dx)>0$ then $\phi_{u,v}$
must have exactly two critical points as
discussed in the remarks preceding the lemma.\\
Thus when $\lambda<\lambda_{1}$ we have obtained a complete
knowledge of the number of critical points of $\phi_{u,v}$, of
the intervals on which $\phi_{u,v}$ is increasing and decreasing
and of the multiples of $(u,v)$ which lie in
$M_{\lambda}(\Omega)$ for every possible choice of signs of
$(\int_{\Omega}b(x)|u|^{\alpha}|v|^{\beta}dx)$ and
$(\int_{\Omega}a(x)|u|^{q}dx+\int_{\Omega}a(x)|v|^{q}dx)$. In
particular we have the following result.
\begin{corollary}
$M_{\lambda}^{0}(\Omega)=\emptyset$ when $0<\lambda<\lambda_{1}$.
\end{corollary}
\begin{corollary} If $\lambda<\lambda_{1}$, then there exists
$\delta_{1}>0$ such that $J_{\lambda}(u,v)\geq\delta_{1}$ for all
$(u,v)\in M_{\lambda}^{-}(\Omega)$.
\end{corollary}
\begin{proof}[\bf Proof]
Consider $(u,v)\in M_{\lambda}^{-}(\Omega)$. Then $\phi_{u,v}$
has a positive global maximum at $t=1$ and
$(\int_{\Omega}b(x)|u|^{\alpha}|v|^{\beta}dx)>0$. Thus
$$\begin{array}{rcl}J_{\lambda}(u,v)&=&\phi_{u,v}(1)\geq\phi_{u,v}(t_{\max})\\\\
&\geq&h_{u,v}(t_{\max})^{\frac{q}{2}}(h_{u,v}(t_{\max})^{\frac{2-q}{2}}-\lambda c)\\\\
&\geq&\delta^{\frac{q}{2}}(\delta^{\frac{2-q}{2}}-\lambda
c)\\\\
\end{array}$$
and the left hand side is uniformly bounded away from $0$
provided that $\lambda<\lambda_{1}$.
\end{proof}

\section{Existence of Positive Solutions}
In this section using the properties of fibering maps we shall
give simple proofs of the existence of two positive solutions,
one in $M^{+}_{\lambda}(\Omega)$ and one in
$M^{-}_{\lambda}(\Omega)$.
\begin{theorem} If
$\lambda<\lambda_{1}$, there exists a minimizer of $J_{\lambda}$
on $M_{\lambda}^{+}(\Omega)$.
\end{theorem}
\begin{proof}[\bf Proof]
Since $J_{\lambda}$ is bounded below on $M_{\lambda}(\Omega)$ and
so on $M_{\lambda}^{+}(\Omega)$, there exists a minimizing
sequence $\{(u_{n},v_{n})\}\subseteq M_{\lambda}^{+}(\Omega)$
such that
$$\displaystyle\lim_{n\rightarrow\infty}J_{\lambda}(u_{n},v_{n})=\displaystyle\inf_{(u,v)\in M_{\lambda}^{+}(\Omega)}J_{\lambda}(u,v).$$
Since $J_{\lambda}$ is coercive, $\{(u_{n},v_{n})\}$ is bounded
in $W$. Thus we may assume, without loss of generality, that
$(u_{n},v_{n})\rightarrow(u_{0},v_{0})$ in $W$ and
$(u_{n},v_{n})\rightarrow(u_{0},v_{0})$ in $L^{r}(\Omega)\times
L^{r}(\Omega)$ for $1<r<\frac{2N}{N-2}$.\\
If we choose $(u,v)\in W$ such that
$(\int_{\Omega}a(x)|u|^{q}dx+\int_{\Omega}a(x)|v|^{q}dx)>0$, then
there exist $t_{1}(u,v)$ such that $(t_{1}(u,v)u,t_{1}(u,v)v)\in
M_{\lambda}^{+}(\Omega)$ and\\
$J_{\lambda}(t_{1}(u,v)u,t_{1}(u,v)v)<0$. Hence,
$\displaystyle\inf_{(u,v)\in M_{\lambda}^{+}(\Omega)
}J_{\lambda}(u,v)<0$. By (2.2),
$$\begin{array}{rcl}J_{\lambda}(u_{n},v_{n})&=&(\frac{1}{2}-\frac{1}{\alpha+\beta})(\int_{\Omega}|\nabla
u_{n}|^{2}dx+\int_{\Omega}|\nabla v_{n}|^{2}dx)\\
&-&\lambda(\frac{1}{\alpha+\beta}-\frac{1}{q})(\int_{\Omega}a(x)|u_{n}|^{q}dx+\int_{\Omega}a(x)|v_{n}|^{q}dx)\\
\end{array}$$ and so
$$\begin{array}{rcl}&&\lambda(\frac{1}{q}-\frac{1}{\alpha+\beta})(\int_{\Omega}a(x)|u_{n}|^{q}dx+\int_{\Omega}a(x)|v_{n}|^{q}dx)=\\
&&(\frac{1}{2}-\frac{1}{\alpha+\beta})(\int_{\Omega}|\nabla
u_{n}|^{2}dx+\int_{\Omega}|\nabla v_{n}|^{2}dx)
-J_{\lambda}(u_{n},v_{n}).\\
\end{array}$$ Letting $n\rightarrow\infty$, we see that
$(\int_{\Omega}a(x)|u_{0}|^{q}dx+\int_{\Omega}a(x)|v_{0}|^{q}dx)>0$.\\

Suppose $(u_{n},v_{n})\nrightarrow(u_{0},v_{0})$ in $W$. We shall
obtain a contradiction by discussing the fibering map
$\phi_{u_{0},v_{0}}$. Hence there exists $t_{0}>0$ such that
$(t_{0}u_{0},t_{0}v_{0})\in M_{\lambda}^{+}(\Omega)$ and
$\phi_{u_{0},v_{0}}$ is decreasing
on $(0,t_{0})$ with $\phi_{u_{0},v_{0}}^{'}(t_{0})=0$.\\
Since $(u_{n},v_{n})\nrightarrow(u_{0},v_{0})$ in $W$,
$$(\int_{\Omega}|\nabla
u_{0}|^{2}dx+\int_{\Omega}|\nabla v_{0}|^{2}dx)<\displaystyle\liminf_{n\rightarrow\infty}(\int_{\Omega}|\nabla
u_{n}|^{2}dx+\int_{\Omega}|\nabla v_{n}|^{2}dx).$$ Thus, as
$$\begin{array}{rcl}\phi_{u_{n},v_{n}}^{'}(t)&=&t(\int_{\Omega}|\nabla
u_{n}|^{2}dx+\int_{\Omega}|\nabla v_{n}|^{2}dx)\\
&-&\lambda t^{q-1}(\int_{\Omega}a(x)|u_{n}|^{q}dx+\int_{\Omega}a(x)|v_{n}|^{q}dx)\\
&-&t^{\alpha+\beta-1}(\int_{\Omega}b(x)|u_{n}|^{\alpha}|v_{n}|^{\beta}dx)\\
\end{array}$$
and
$$\begin{array}{rcl}\phi_{u_{0},v_{0}}^{'}(t)&=&t(\int_{\Omega}|\nabla
u_{0}|^{2}dx+\int_{\Omega}|\nabla v_{0}|^{2}dx)\\
&-&\lambda t^{q-1}(\int_{\Omega}a(x)|u_{0}|^{q}dx+\int_{\Omega}a(x)|v_{0}|^{q}dx)\\
&-&t^{\alpha+\beta-1}(\int_{\Omega}b(x)|u_{0}|^{\alpha}|v_{0}|^{\beta}dx).\\
\end{array}$$
It follows that $\phi_{u_{n},v_{n}}^{'}(t_{0})>0$ for $n$
sufficiently large. Since $\{(u_{n},v_{n})\}\subseteq
M_{\lambda}^{+}(\Omega)$, by considering the possible fibering
maps it is easy to see that $\phi_{u_{n},v_{n}}^{'}(t)<0$ for
$0<t<1$ and $\phi_{u_{n},v_{n}}^{'}(1)=0$ for all $n$. Hence we
must have $t_{0}>1$. But  $(t_{0}u_{0},t_{0}v_{0})\in
M_{\lambda}^{+}(\Omega)$ and so
$$J_{\lambda}(t_{0}u_{0},t_{0}v_{0})<J_{\lambda}(u_{0},v_{0})<\displaystyle\lim_{n\rightarrow\infty}J_{\lambda}(u_{n},v_{n})=\displaystyle\inf_{(u,v)\in
M_{\lambda}^{+}(\Omega)}J_{\lambda}(u,v)$$ and this is a
contradiction. Hence $(u_{n},v_{n})\rightarrow(u_{0},v_{0})$ in
$W$ and so
$$J_{\lambda}(u_{0},v_{0})=\displaystyle\lim_{n\rightarrow\infty}J_{\lambda}(u_{n},v_{n})=\displaystyle\inf_{(u,v)\in
M_{\lambda}^{+}(\Omega)}J_{\lambda}(u,v).$$ Thus $(u_{0},v_{0})$
is a minimizer for $J_{\lambda}$ on $M_{\lambda}^{+}(\Omega)$.
\end{proof}
\begin{theorem} If $\lambda<\lambda_{1}$, then exists a minimizer of
$J_{\lambda}$ on $M_{\lambda}^{-}(\Omega)$.
\end{theorem}
\begin{proof}[\bf Proof] By Corollary 3.3 we have
$J_{\lambda}(u,v)\geq\delta_{1}>0$ for all $(u,v)\in
M_{\lambda}^{-}(\Omega)$ and so $\displaystyle\inf_{(u,v)\in
M_{\lambda}^{-}(\Omega)}J_{\lambda}(u,v)\geq\delta_{1}$. Hence
there exists a minimizing sequence $\{(u_{n},v_{n})\}\subseteq
M_{\lambda}^{-}(\Omega)$ such that
$$\displaystyle\lim_{n\rightarrow\infty}J_{\lambda}(u_{n},v_{n})=\displaystyle\inf_{(u,v)\in
M_{\lambda}^{-}(\Omega)}J_{\lambda}(u,v)>0.$$ As in the previous
proof, since $J_{\lambda}$ is coercive, $\{(u_{n},v_{n})\}$ is
bounded in $W$ and we may assume, without loss of generality,
that $(u_{n},v_{n})\rightarrow(u_{0},v_{0})$ in $W$ and
$(u_{n},v_{n})\rightarrow(u_{0},v_{0})$ in $L^{r}(\Omega)\times
L^{r}(\Omega)$ for $1<r<\frac{2N}{N-2}$. By (2.2)
$$\begin{array}{rcl}J_{\lambda}(u_{n},v_{n})&=&(\frac{1}{2}-\frac{1}{q})(\int_{\Omega}|\nabla
u_{n}|^{2}dx+\int_{\Omega}|\nabla v_{n}|^{2}dx)\\
&+&(\frac{1}{q}-\frac{1}{\alpha+\beta})(\int_{\Omega}b(x)|u_{n}|^{\alpha}|v_{n}|^{\beta}dx)\\
\end{array}$$
and, since
$\displaystyle\lim_{n\rightarrow\infty}J_{\lambda}(u_{n},v_{n})>0$
and
$$\displaystyle\lim_{n\rightarrow\infty}(\int_{\Omega}b(x)|u_{n}|^{\alpha}|v_{n}|^{\beta}dx)=(\int_{\Omega}b(x)|u_{0}(x)|^{\alpha}|v_{0}(x)|^{\beta}dx),$$
we must have that
$(\int_{\Omega}b(x)|u_{0}(x)|^{\alpha}|v_{0}(x)|^{\beta}dx)>0$.
Hence there exist $\hat{t}>0$ such that
$(\hat{t}u_{0},\hat{t}v_{0})\in
M_{\lambda}^{-}(\Omega)$.\\

Suppose $(u_{n},v_{n})\nrightarrow(u_{0},v_{0})$ in $W$. Using the
facts that
$$(\int_{\Omega}|\nabla
u_{0}|^{2}dx+\int_{\Omega}|\nabla v_{0}|^{2}dx)<\displaystyle\liminf_{n\rightarrow\infty}(\int_{\Omega}|\nabla
u_{n}|^{2}dx+\int_{\Omega}|\nabla v_{n}|^{2}dx)$$ and that, since
$(u_{n},v_{n})\in M_{\lambda}^{-}(\Omega)$,
$J_{\lambda}(u_{n},v_{n})\geq J_{\lambda}(su_{n},sv_{n})$ for all
$s\geq0$, we have
$$\begin{array}{rcl}J_{\lambda}(\hat{t}u_{0},\hat{t}v_{0})&=&\frac{1}{2}\hat{t}^{2}(\int_{\Omega}|\nabla
u_{0}|^{2}dx+\int_{\Omega}|\nabla v_{0}|^{2}dx)\\
&-&\frac{\lambda\hat{t}^{q}}{q}(\int_{\Omega}a(x)|u_{0}|^{q}dx+\int_{\Omega}a(x)|v_{0}|^{q}dx)\\
&-&\frac{\hat{t}^{\alpha+\beta}}{\alpha+\beta}(\int_{\Omega}b(x)|u_{0}|^{\alpha}|v_{0}|^{\beta}dx)\\
&<&\displaystyle\lim_{n\rightarrow\infty}[\frac{1}{2}\hat{t}^{2}(\int_{\Omega}|\nabla
u_{n}|^{2}dx+\int_{\Omega}|\nabla v_{n}|^{2}dx)\\
&-&\frac{\lambda\hat{t}^{q}}{q}(\int_{\Omega}a(x)|u_{n}|^{q}dx+\int_{\Omega}a(x)|v_{n}|^{q}dx)\\
&-&\frac{\hat{t}^{\alpha+\beta}}{\alpha+\beta}(\int_{\Omega}b(x)|u_{n}|^{\alpha}|v_{n}|^{\beta}dx)]\\
&=&\displaystyle\lim_{n\rightarrow\infty}J_{\lambda}(\hat{t}u_{n},\hat{t}v_{n})\\
&\leq&\displaystyle\lim_{n\rightarrow\infty}J_{\lambda}(u_{n},v_{n})=\displaystyle\inf_{(u,v)\in M_{\lambda}^{-}(\Omega)}J_{\lambda}(u,v)\\
\end{array}$$ which is a contradiction. Hence $(u_{n},v_{n})\rightarrow(u_{0},v_{0})$ in
$W$ and the proof can be completed as in the previous theorem.
\end{proof}
\begin{corollary} System $(E_{\lambda})$ has at least two positive
solutions whenever \\$0<\lambda<\lambda_{1}$.
\end{corollary}
\begin{proof}[\bf Proof]
By Theorems 4.1 and 4.2 there exist $(u^{+},v^{+})\in
M_{\lambda}^{+}(\Omega)$ and $(u^{-},v^{-})\in
M_{\lambda}^{-}(\Omega)$ such that
$$J_{\lambda}(u^{+},v^{+})=\displaystyle\inf_{(u,v)\in
M_{\lambda}^{+}(\Omega)}J_{\lambda}(u,v)$$ and
$$J_{\lambda}(u^{-},v^{-})=\displaystyle\inf_{(u,v)\in
M_{\lambda}^{-}(\Omega)}J_{\lambda}(u,v).$$ Moreover
$J_{\lambda}(u^{\pm},v^{\pm})=J_{\lambda}(|u^{\pm}|,|v^{\pm}|)$
and $(|u^{\pm}|,|v^{\pm}|)\in M^{\pm}_{\lambda}(\Omega)$ and so
we may assume $u^{\pm}\geq0$ and $v^{\pm}\geq0$. By Lemma 2.2
$(u^{\pm},v^{\pm})$ are critical points of $J_{\lambda}$ on $W$
and hence are weak solutions (and so by standard regularity
results classical solutions) of $(E_{\lambda})$. Finally, by the
Harnack inequality due to Trudinger \cite{21}, we obtain that
$(u^{\pm},v^{\pm})$ are positive solutions of $(E_{\lambda})$.
\end{proof}
%----------------------------------------------------------------------------------------%

\end{document}